\documentclass[reqno,10pt]{amsart}
\usepackage{amsmath,amsthm,amsfonts,color,graphicx,soul}
\usepackage[latin1]{inputenc}
\usepackage[makeroom]{cancel}
\usepackage{tikz}

\usetikzlibrary{decorations.pathreplacing}

\newtheorem{thm}{Theorem}[section]
\newtheorem{lem}[thm]{Lemma}

\newtheorem{prop}[thm]{Proposition}

\newtheorem{rem}[thm]{Remark}

\numberwithin{equation}{section}

\definecolor{mygreen}{RGB}{228,0,0}

\DeclareMathAlphabet{\mathpzc}{OT1}{pzc}{m}{it}

\newcommand{\Z}{\mathbb Z}
\newcommand{\R}{\mathbb R}

\begin{document}
\bibliographystyle{plain}

\title[On a model for a sliding droplet]{On a model for a
    sliding droplet:\\ Well-posedness  and stability of  translating
    \\ circular solutions}

\author{Patrick Guidotti}
\address{University of California, Irvine\\
Department of Mathematics\\
340 Rowland Hall\\
Irvine, CA 92697-3875\\ USA }
\email{gpatrick@math.uci.edu}

\author{Christoph Walker}
\address{Leibniz Universit\"at Hannover\\
Institut f\"ur Angewandte Mathematik\\
Welfengarten 1\\
30167 Hannover\\
Germany}
\email{walker@ifam.uni-hannover.de}

\begin{abstract}
In this paper the model for a highly viscous droplet sliding down an
inclined plane is analyzed. It is shown that, provided the slope is
not too steep, the corresponding moving boundary problem possesses
classical solutions. Well-posedness is lost when the relevant
linearization ceases to be parabolic.  This occurs above a critical
incline which depends on the shape of the initial wetted region as
well as on the liquid's mass. It is also
shown that translating circular solutions are asymptotically stable 
if the kinematic boundary condition is given by an affine function
and provided the incline is small.
\end{abstract}


\maketitle

\section{Introduction}
Of interest is the analysis of a model for the motion under gravity of a highly
viscous droplet on an  inclined homogeneous substrate as
depicted below.\\
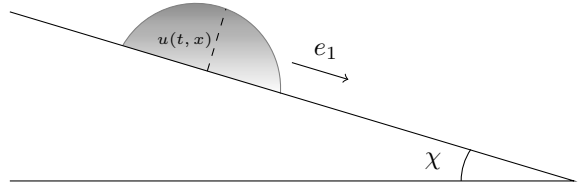
\begin{center}
\begin{tikzpicture}[scale=1.5]
  \draw[shade,gray] (-4,1.2) arc (150:-4:0.75);
  \draw (-5,0)--(0,0)--(-5,1.5);
  \draw (-1,0) arc (180:146:0.5);
  \node at (-1.25,0.175) {$\chi$};
  \draw[->]  (-2.5,1.05)--(-2,.9);
  \node at (-2.2,1.15) {$e_1$};
  \node at (0,-0.3) {};
  \draw[dashed] (-3.25,{39/40})--({-3.25+1.1*.15},{39/40+1.1*.5});
  \node at (-3.45,1.25) {\tiny$u(t,x)$};
\end{tikzpicture}
\end{center}
The droplet is characterized by the wetted region $\Omega (t)=[u(t,\cdot)>0]$ on the
substrate and a height field $u(t,x)$ measured in direction normal to
the plane of motion at points $x\in \Omega(t)$, as
  depicted above. It follows that
$u(t,x)=0$ for $x\in \Gamma(t):=\partial \Omega(t)$.  
The system reads
 \begin{alignat}{2}\label{d1}
  -\Delta u&=\mu x^1+\lambda&&\text{in }\Omega(t),\\
 \label{d2}
  u&=0&&\text{on }\Gamma(t),\\\label{d3}
  \int _{\Omega(t)}u\, dx&=\mathcal{V},&&\\\label{d4}
  V&=F(-\partial _\nu u)\qquad&&\text{on }\Gamma(t),\\\label{d5}
  \Omega(0)&=\Omega^0,&&
 \end{alignat}
where the parameter $\mu\geq 0$ is related to the inclination
$\chi\geq 0$ via $\mu=\mathrm{Bo}\sin\chi$, where $\mathrm{Bo}$ is
the Bond number representing the relative magnitude
of gravity and viscous forces. The unknown $\lambda$
can be thought of as a time-dependent Lagrange multiplier for the
third condition enforcing conservation of the total volume
$\mathcal{V}>0$ (actually, it is a
constant resulting from integrating the original fourth-order equation
for $u$, see \cite{BCP}). The vector $\nu$ is 
the outward unit normal to the boundary $\Gamma(t)$ of the wetted region
and $V$ is the speed in normal direction of the same
boundary. Coordinates $x=(x^1,x^2)\in \mathbb{R}^2$ are used to
describe points on the inclined plane, where $x^1$ is the coordinate
in the direction of motion. Observe that  the location of the origin
of the coordinate system is irrelevant as any translation
in the direction of $e_1$ is absorbed by the Lagrange
  multiplier $\lambda$ and the problem is invariant under
translations in the direction of $e_2$.
Equation~\eqref{d4} corresponds to a kinematic boundary condition
relating the normal velocity to the (small) dynamic contact angle
through an empirical law described by the function $F$.
The dimensionless equations  \eqref{d1}--\eqref{d5} for an inclined
plane were derived in \cite{BCP} under the assumption that
a lubrication approximation is applicable and the
Navier-Stokes equations thus greatly simplify. Actually, the model for
a droplet on a horizontal plane with $\chi=0$ and hence $\mu=0$ was
introduced some time ago in \cite{Gr78}. For numerical experiments and
numerical schemes in this case we refer to \cite{HM} and \cite{Gl},
respectively. For this case local and global existence results for
generalized weak solutions are to be found in
\cite{KGl,KGr}. Moreover, short time existence of classical solutions
was proved in \cite{EG} while in \cite{PG1} circles were
identified as the only equilibria and shown to be locally asymptotically stable. 

The situation for an inclined plane with $\chi>0$ (and
hence $\mu>0$) considered herein is somewhat different. Indeed, in
\cite{BCP} it was shown that for an affine\footnote{This means that
  the normal velocity is proportional to the difference between
  equilibrium and dynamic contact angle for the fluid-substrate
  system. Such a form was derived in \cite{BCP} as the linearization
  for small contact angles of the ``Cox-Voinov law'', a simple
  particular choice of the many available laws for $F$, see
  \cite{PFL,BCP}.} function $F$ there is a 
critical inclination of the substrate below which a translating circular
solution to \eqref{d1}--\eqref{d5} exists moving at a constant
speed, while such a solution ceases to exist if the incline is
increased any further. More precisely, consider a 
droplet sliding down the substrate at a constant speed
  $v_0>0$, so that the wetted region is of the form
$\Omega(t)=\Omega_*+tv_0e_1$ with normal velocity  $V=v_0 e_1\cdot
\nu$, where $e_1$ is the unit vector in $x^1$-direction and
$u(t,x)=u_0(x-tv_0e_1)$. Then it is readily seen that 
\eqref{d1}--\eqref{d4} is equivalent to
 \begin{alignat}{2}\label{d1b}
  -\Delta u_0&=\mu x^1+\lambda&&\text{in }\Omega_*,\\
 \label{d2b}
  u_0&=0&&\text{on }\Gamma_*,\\\label{d3b}
  \int _{\Omega_*}u_0\, dx&=\mathcal{V},&&\\\label{d4b}
 F( -\partial _\nu u_0)&=v_0 e_1\cdot \nu\qquad&&\text{on }\Gamma_*,
 \end{alignat}
for the unknowns $u_0$, $\lambda$, and $\Omega_*$ with
$\Gamma_*=\partial\Omega_*$. Using polar coordinates, set
\begin{equation}\label{ui}
u_0(r,\theta):=\frac{\mu r}{8} (R_0^2-r^2)\cos \theta+\frac{2
    \mathcal{V}}{\pi R_0^4}(R_0^2-r^2),\quad r\in [0,R_0],\quad
\theta\in [0,2\pi],  
\end{equation}
which is easily checked to be the (unique) solution to
\eqref{d1b}--\eqref{d3b} on the disk $\Omega_*=R_0\mathbb{B}$. Here 
$\mathbb{B}=\mathbb{B}_{\mathbb{R}^2}$ denotes the two-dimensional
unit ball centered at the origin. Moreover, \eqref{d4b} becomes  
\begin{equation}\label{N}
F\left(\frac{\mu R_0^2}{4}\cos\theta+\frac{4\mathcal{V}}{\pi
  R_0^3}\right)=v_0\cos\theta,\quad \theta\in [0,2\pi]. 
\end{equation}
If $\mu=0$, that is, if the incline vanishes, then \eqref{N} implies
that $v_0$ must vanish and $F$ must posses a zero at 
  $4\mathcal{V}/\pi R_0^3$. In particular, no translating solution
can exist on a disk if $\mu=0$. If $\mu>0$, then
  \eqref{N} implies that translating solutions only exist provided
that $F$ is affine, i.e. 
\begin{equation}\label{affine}
F(q)=aq-b,\quad q\in [-v_0,v_0],\qquad a, b>0,
\end{equation}
as in \cite{BCP}. Therefore, if $F$ is affine, then there are a unique
radius   and a unique velocity, given by 
\begin{equation}\label{R0}
R_0=\left(\frac{4 \mathcal{V}a}{\pi b}\right)^{1/3},\qquad
v_0=\frac{\mu a}{4}\left(\frac{4 {\mathcal{V}}a}{\pi
    b}\right)^{2/3}, 
\end{equation}
such that \eqref{N} holds. While the velocity is of order $\mu$, the
radius is independent of the incline. 
Consequently, if $\mu>0$, then \eqref{affine} is a sufficient, but also
necessary condition for a disk to be 
a translating geometry solution of \eqref{d1}--\eqref{d5}, that is,
for the existence a (unique) 
solution $u_0$ to problem \eqref{d1b}--\eqref{d4b} on a disk 
$\Omega_*=R_0\mathbb{B}$ which is then given by
\eqref{ui} with uniquely determined $R_0$ and $v_0$ in dependence
of $a$ and $b$. However, to guarantee the positivity of
$u_0$ in $\Omega_*$ we need $16\mathcal{V}\ge \mu \pi R_0^5$ as is seen by taking
$\theta=\pi$ in \eqref{ui}. Writing $\mu=\mathrm{Bo} \sin\chi$ we
derive the physical restriction on the maximal inclination of the
substrate as 
$$
\sin\chi\le \frac{16 \mathcal{V}}{\pi \mathrm{Bo} R_0^5}. 
$$
Therefore, circular solutions to \eqref{d1}-\eqref{d4} only exist if
either $\mu=0$ and $F$ has a zero or if $\mu>0$ is sufficiently small
and $F$ is an affine function. It is an interesting question to
determine whether non-circular translating solutions exist for
general functions $F$. Observe that this cannot be the case for
$\mu=0$ by Serrin's Rigidity Theorem \cite{Serrin}. For $\mu>0$,
however, the assumptions of  Serrin's theorem are not met since the
right-hand sides of \eqref{d1b} and \eqref{d4b} are no longer
constant. Uniqueness of (non-circular) translating solutions is therefore not
clear. 

In this paper we establish the local well-posedness 
  of \eqref{d1}--\eqref{d5} for general initial droplet geometries
$\partial\Omega_0$ and for general laws $F$ for inclines
  $\chi$ smaller than a positive critical value, that is, we do not impose
any structural conditions on $F$ except that
\begin{equation}\label{F}
F\in C^{4}(\mathbb{R},\mathbb{R}),\quad F'>0.
\end{equation}
In contrast to the often used Hanzawa transformation (e.g. see
\cite{ES, EG, PG1} and the references therein) we shall use to this
end a description of the unknown curve $\Gamma(t)$ by means of
coordinates induced by a smooth flow transversal to the curve
$\Gamma_0=\partial\Omega_0$ (see Section~\ref{Sec2}). 
This approach to moving boundary problems was first introduced in
\cite{PG2} in a more general context and the analysis performed in
this paper therefore provides a concrete demonstration of the benefits
that it offers. In particular, it yields a significant simplification
of the analysis required. In short, problem \eqref{d1}--\eqref{d5} is
reduced to a single nonlocal, nonlinear evolution equation and yields
a relatively simple and insightful formula for its linearization (see
Theorem~\ref{T1}). The latter, not only, is the basis for using
maximal regularity results to obtain local well-posedness (see
Theorem~\ref{T2}), but also for the characterization of the critical
incline beyond which the model ceases to be parabolic in nature and
becomes ill-posed. Finally, we also investigate the stability of
the sliding circular droplet, the existence of which is
  ensured by \eqref{affine}. We show that, for small
  inclines and when starting out with an initial droplet geometry sufficiently close to
the circle of radius $R_0$, the droplet asymptotically becomes circular
of radius $R_0$ sliding down the plane with constant
speed $v_0$  (see Theorem~\ref{T4}).

\section{Reformulation}\label{Sec2}

{System \eqref{d1}--\eqref{d5} can be reduced to a nonlocal geometric
evolution for the unknown closed curve $\Gamma(t)$ enclosing the
simply connected domain $\Omega(t)$. The derivation of a suitable
description of this evolution is the purpose of this section.

Working in a
classical regularity context, we consider domains $\Omega$ with boundary $\partial\Omega$ of class
$buc^{2+\alpha}$ with $\alpha\in (0,1)$, i.e. domains the boundary of which
are locally the graph of functions belonging to the little H\"older
space $buc^{2+\alpha}$. Recall that, for an open subset $O$ of
$\mathbb{R}^n$, the space $\operatorname{BUC}^{2+\alpha}(O)$, defined
by 
\begin{equation*}
\begin{split}
 \operatorname{BUC}^{2+\alpha}(O):=\big\{ u:O\rightarrow
 \mathbb{R}\,\big |\, &\partial^\beta u\text{ is uniformly 
   continuous and bounded for }|\beta
 |\le 2 \\
&\text{ and } [\partial^\beta u]_\alpha<\infty\text{ for }|\beta
 |=2\big\},
\end{split}
\end{equation*}
is a Banach space with respect to the norm
$$
 \| u\|_{2,\alpha}:=\| u\| _{2,\infty}+\max_{|\beta
   |=2}[\partial^\beta u]_\alpha,\qquad u\in
 \operatorname{BUC}^{2+\alpha}(O),
$$
for
$$
 \| u\| _{2,\infty}=\sum_{|\gamma|\leq 2}\| \partial ^\gamma u\| _\infty \quad
 \text{ and }\quad [u]_\alpha :=\sup _{x\neq y}\frac{|u(x)-u(y)|}{|x-y|^\alpha}.
$$
Then $buc^{2+\alpha}(O)$ is the closure of
$\operatorname{BUC}^{2+\alpha+\epsilon}(O)$ in
$\operatorname{BUC}^{2+\alpha}(O)$ for any $\epsilon>0$. Given any
smooth closed curve $\Gamma\in buc^{2+\alpha}$, the space
$buc^{2+\alpha}(\Gamma)$ can be defined in the standard way via local
charts and a partition of unity. 

In the rest of the paper we use the
notation $x\cdot y$ and $(x| y)$ interchangeably for the scalar
product of vectors $x,y\in\mathbb{R}^2$ to enhance the readability of
some formul\ae. 

For a fixed domain $\Omega$ with boundary $\Gamma\in buc^{2+\alpha}$ we
first solve the sub-problem \eqref{d1}--\eqref{d3} to obtain a 
solution $\bigl( u   ,\lambda  \bigr)$.}
We will use the notation $u^{(f)}$ to indicate the
solution of $-\Delta u=f$ in $\Omega$ which vanishes on the boundary
$\Gamma=\partial\Omega$. 

\begin{prop}\label{P1}
For any domain $\Omega$ bounded by a closed curve $\Gamma\in buc^{2+\alpha}$ and any $\mu\ge 0$,
problem \eqref{d1}--\eqref{d3} has a unique solution
$$
( u ,\lambda)\in buc^{2+\alpha}(\Omega)\times
 \mathbb{R} 
$$
with
\begin{equation}\label{lambda}
  \lambda=\frac{{\mathcal{V}}}{\int_\Omega u^{(1)}\, dx}-\mu\frac{\int _\Omega
    u^{(x^1)}\, dx}{\int_\Omega u^{(1)}\, dx}.
\end{equation}
There are $0<\mu_0\le\mu_1$ depending on $\Omega$ and ${
  \mathcal{V}}$ such that {$u>0$ in $\Omega$ and}
$-\partial_{\nu}u>0$ on $\Gamma$ if $\mu\in [0,\mu_0)$ and
$-\partial_{\nu}u<0$ on $\Gamma$ if $\mu\ge\mu_1$. 
\end{prop}
\begin{proof}
Since the right-hand side of \eqref{d1} is smooth, classical theory
for elliptic boundary value problems ensures that
\eqref{d1}--\eqref{d2} has, for any fixed $\lambda$, a unique
solution $u(\lambda)\in buc^{2+\alpha}(\Omega)$ and
$$
 u(\lambda)=\mu u^{(x^1)} +\lambda u^{(1)},
$$
by linearity of the equation. The unknown parameter $\lambda$ is then
determined by \eqref{d3} from
$$
 \mu\int _\Omega u^{(x^1)} \, dx+\lambda\int _\Omega u^{(1)}\,
 dx={\mathcal{V}}, 
$$
which yields {formula \eqref{lambda}. Next note that we can always fix
  the origin in such a way that $x^1$ is negative throughout the
  domain $\Omega$. Thus, since $u^{(1)} >0$ in $\Omega$ by the maximum
  principle, \eqref{lambda} implies that there are $\mu_1\ge \mu_0>0$
  such that $\lambda +\mu x^1\ge 0$ in $\Omega$ for $\mu\in (0,\mu_0)$
  while  $\lambda +\mu x^1\le 0$ in $\Omega$ for $\mu\ge
  \mu_1$. Consequently, if $\mu\in [0,\mu_0)$, then $u>0$ in $\Omega$
  by the maximum principle and $\partial_{\nu}u<0$ on $\Gamma$ by
  Hopf's Lemma, while if $\mu\ge \mu_1$, then  $-\partial_{\nu}u<0$ on
  $\Gamma$.}
\end{proof}

With $u$ in hand, equations
\eqref{d4}--\eqref{d5} amount to a nonlocal (geometric) evolution for
the closed curve $\Gamma(t)$. In order to obtain an evolution equation
for it, it is necessary to gain a local understanding of the manifold
$\mathcal{M}^{2+\alpha}$ of closed curves {in $\mathbb{R}^2$} of class $buc^{2+\alpha}$. A
convenient local parametrization about a given fixed initial curve $\Gamma_0$
is particularly useful. Let $\tau_0$, $\nu_0$ be
the unit tangent and normal vectors for $\Gamma_0$, the latter
pointing out of the domain enclosed by $\Gamma_0$.

\begin{lem}[Tubular Neighborhood]\label{tub}
Given $\Gamma_0\in \mathcal{M}^{2+\alpha}$, there is $r_0>0$
such that
$$
 T_{r_0}(\Gamma_0):=\{ x\in \mathbb{R}^2\, |\, |d(x,\Gamma_0)|<r_0\}
$$
is an open neighborhood of $\Gamma_0$ diffeomorphic to
$\Gamma_0\times(-r_0,r_0)$. The notation $d(\cdot,\Gamma_0)$ is used
for the signed distance function to $\Gamma_0$.
\end{lem}

\begin{proof}
Define the map
$$
 \Phi: \Gamma_0\times(-r_0,r_0) \longrightarrow \mathbb{R}^2,\:
 (y,r)\mapsto y+r\nu_0(y),
$$
and recall that, if $\gamma_0=\gamma_0(s)$ is an arc length
parametrization of $\Gamma_0$ about the point $y\in \Gamma_0$, then
$$
 \gamma_0'(s)=\tau_0(\gamma_0(s))\text{ and }
 d_s\nu_0(\gamma_0(s))=\kappa_0(\gamma_0(s))\tau_0(\gamma_0(s)), 
$$
where $\kappa_0$ denotes the curvature of $\Gamma_0$.
Hence, computing in the above coordinates,
$$
 d_y\Phi(y,r)=\tau_0(y)+r\kappa_0(y)\tau_0(y)\text{ and }
 d_r\Phi(y,r)=\nu_0(y). 
$$
The assumption on $\Gamma_0$ implies that
$$
 \| \kappa _0\| _\infty<\infty,
$$
and thus that $d\Phi(y,r)$ is invertible at least as long as
$1+r\kappa_0(y)>0$, since its columns are orthogonal. This can be
ensured by choosing $r_0<1/\| \kappa _0\| _\infty$ and the inverse
function theorem yields local invertibility of $\Phi$. Assuming
without loss of generality that $r_0$ is also so small that
$$
 \mathbb{B}_{\Gamma_0}(y,2r_0)\times(-2r_0,2r_0)\cap
 \Gamma_0=\mathbb{B}_{\Gamma_0}(y,2r_0),
$$
for each $y\in \Gamma_0$, it follows that $x\in T_{r_0}(\Gamma_0)$
has a unique representation by $\Phi$, yielding global injectivity.
\end{proof}

\begin{rem}
The map $\Phi$ yields a foliation of $T_{r_0}(\Gamma_0)$ by
$buc^{1+\alpha}$ curves since it uses $\nu _0\in
buc^{1+\alpha}(\Gamma_0)$. For technical reasons this regularity is not
sufficient.
\end{rem}

\begin{lem}[Generalized Tubular Neighborhood]\label{gentub}
Given $\Gamma_0\in \mathcal{M}^{2+\alpha}$, there is $r_0>0$ and
curves $\Gamma_r\in \mathcal{M}^{2+\alpha}$ for $r\in(-r_0,r_0)$ such
that 
$$
 \bigcup _{|r|<r_0}\Gamma_r
$$
is an open neighborhood of $\Gamma_0$ that is diffeomorphic to
$
 \Gamma_0\times(-r_0,r_0).
$
\end{lem}

\begin{proof}
By the preceding lemma, there exists $\tilde r_0$ such that, for
  each $x\in T_{\tilde r_0}(\Gamma_0)$, there is a unique $(y,r)=\bigl( 
 y(x),r(x) \bigr)\in \Gamma_0\times (-\tilde r_0,\tilde r_0)$
with $x=y+r\nu_0(y)$. In $T_{\tilde r_0}(\Gamma_0)$ define the field
$$
 \tilde{\tilde{\nu}}(x):=\nu_0 \bigl( y(x)\bigr),\quad x\in T_{\tilde r_0}(\Gamma_0),
$$
take a smooth cutoff function $\eta:\mathbb{R}\to \mathbb{R}$ with
$$
 0\leq\eta\leq 1,\quad \eta \big |_{[-\tilde r_0/2,\tilde r_0/2]}\equiv
 1, \quad\text{ and }\ \eta \big |_{(-3\tilde r_0/4,3\tilde r_0/4)^\mathsf{c}}\equiv
 0,
$$
and set
$$
 \tilde \nu(x):=
 \begin{cases}
   \tilde{\tilde{\nu}}(x)\eta \bigl( r(x)\bigr) ,&x\in T_{\tilde r_0}(\Gamma_0),\\
   0,&x\notin T_{\tilde r_0}(\Gamma_0).
 \end{cases}
$$
Then $\tilde\nu\in buc^{1+\alpha}(\mathbb{R}^2,\mathbb{R}^2)$ is a
global vector field. Now take a compactly supported smooth mollifier
$\psi_\delta$ and define
$$
 \nu^\delta :=\psi_\delta *\tilde\nu
$$
componentwise. It follows that $\nu^\delta\in
\operatorname{BUC}^\infty(\mathbb{R}^2,\mathbb{R}^2)$ and that
\begin{equation}\label{normreg}
 \nu^\delta \longrightarrow \nu_0\circ y\ \text{ in } \ buc^{1+\alpha}\bigl(
 T_{\tilde r_0/2}(\Gamma_0)\bigr)
\end{equation}
as $\delta\to 0$. In particular it holds that
$$
 |\nu^\delta (y)\cdot\tau_0(y)|\leq c(\delta)<1,\: y\in \Gamma_0,
$$
where $c(\delta)\to 0$ as $\delta\to 0$.
The
vector field is therefore uniformly transversal to $\Gamma_0$ for
$\delta>0$ fixed small enough. Let then
$\varphi^\delta=\varphi^\delta(y,r)$ be the flow generated by the ODE
$$\begin{cases}
 \dot x = \frac{d}{dr} x=\nu^\delta(x),&\\
 x(0)=y\in \Gamma_0.&
\end{cases}
$$  
It is easily seen that there is $r_0>0$ such that $\varphi^\delta$ is
defined on $\Gamma_0\times (-r_0,r_0)$ with
$$
  \Gamma_r:= \varphi^\delta (\Gamma_0,r)\subset T_{\tilde
   r_0}(\Gamma_0),\ |r|<r_0
$$
and standard ODE arguments yield that
$$
 \varphi ^\delta :\Gamma_0\times (-r_0,r_0)\longrightarrow
 \bigcup_{|r|<r_0}\Gamma_r 
$$
is a diffeomorphism of class $buc^{2+\alpha}$.
\end{proof}
In the following we use the notation introduced in the proof of
Lemma~\ref{gentub} for a fixed $\Gamma_0\in \mathcal{M}^{2+\alpha}$,
that is, having $\delta>0$ chosen small enough we let $\varphi^\delta:
\Gamma_0\times (-r_0,r_0)\rightarrow \mathbb{R}^2$ denote the flow
induced by the vector field $\nu^\delta{\in
\operatorname{BUC}^\infty(\mathbb{R}^2,\mathbb{R}^2)}$ transversal to $\Gamma_0$. It
is convenient to define 
$$
\varphi^\delta_r(y):=\varphi^\delta(y,r)
,\qquad \nu^\delta_r(y):=\nu^\delta\big(\varphi^\delta_r(y)\big)
$$
for $(y,r)\in \Gamma_0\times (-r_0,r_0)$, hence 
\begin{equation}\label{flow}
\frac{d}{dr} \varphi^\delta_r(y) = \nu^\delta_r(y), \quad (y,r)\in
\Gamma_0\times (-r_0,r_0). 
\end{equation}
Note that $\nu^\delta_0=\nu^\delta$ since
$\varphi^\delta_0=\operatorname{id}$.
The previous lemma provides coordinates $(y,r)$  for a neighborhood of
$\Gamma_0$, which can be denoted by $T^{\nu^\delta}_{r_0}(\Gamma_0)$
since it is constructed based on $\nu ^\delta$. Explicitly this means
\begin{equation}\label{coord}
 \forall\, x\in T^{\nu^\delta}_{r_0}(\Gamma_0)\ \exists!\, (y,r)\in
 \Gamma_0\times(-r_0,r_0)\text{ such that }
 x=\varphi^\delta_r(y).
\end{equation}
We show next that curves close to $\Gamma_0$ can be conveniently
parametrized in these coordinates. 

\begin{lem}\label{L25}
Let $\Gamma_0$ and $ \Gamma\in \mathcal{M}^{2+\alpha}$ be
close in the $\operatorname{BUC}^1$-topology, that is,
  let they satisfy
$$
 d_H \Bigl( \big\{ (\tilde y,\nu_{\Gamma} \bigl( \tilde y)\bigr) \,
 \big |\,\tilde y\in \Gamma\big\} ,\big\{ \bigl( y,\nu_0(y)\bigr) \,
 \big |\, y\in \Gamma_0\big\} \Bigr) <<1, 
$$
where $d_H$ denotes the Haussdorff distance between compact sets. Then
there is a unique function $\rho\in
buc^{2+\alpha}(\Gamma_0)$ such that
$$
 \Gamma=\big\{ \varphi^\delta \bigl( y,\rho(y)\bigr) \, \big |\, y\in
 \Gamma_0\big\}= \varphi^\delta _\rho(\Gamma_0).
$$
\end{lem}

\begin{proof}
We refer to \cite[Lemma 2.6]{PG2} for a complete proof.
\end{proof}

\section{The Equation for $\Gamma(t)$ and its Linearization}\label{Sec3}

We now focus on equations \eqref{d4}--\eqref{d5} for  a given
simply connected domain  {$\Omega^0$} of class $buc^{2+\alpha}$ with
$\alpha\in (0,1)$, that is, $\Gamma_0:=\partial{ \Omega^0}\in \mathcal{M}^{2+\alpha}$. 
Using the corresponding coordinates introduced in \eqref{coord} {with $\delta>0$ small enough}, it is possible to reduce the
evolution for $\Gamma(t)$ to one for the function
$\rho(t,\cdot):\Gamma_0\to \mathbb{R}$ given in Lemma~\ref{L25} through
$$
 \Gamma(t)=\big\{ \varphi^\delta_{\rho(t,y)}(y)\, \big |\, y\in
 \Gamma_0\big\}= \varphi^\delta_{\rho(t,\cdot)}(\Gamma_0)=:
 \Gamma_{\rho(t)}.
$$
Denote the unit tangent and normal vectors to $\Gamma_{\rho(t)}$ at
the point $\varphi^\delta_{\rho(t,y)}(y)$ by $\tau_{\rho(t)}(y)$ and
$\nu_{\rho(t)}(y)$, respectively. Then \eqref{flow} implies
$$
\frac{d}{dt}\varphi^\delta_{\rho(t,y)} ( y)=  \nu^\delta_{\rho(t,y)}
(y)\dot \rho(t,y), \quad y\in \Gamma_0, 
$$
with superscript dot indicating a derivative with respect to
time. Since the normal velocity $V_{\Gamma_{\rho(t)}}$ of
$\Gamma_{\rho(t)}$ at a point $\varphi^\delta_{\rho(t,y)} ( y)$ is
given by the component of the projection of the tangent vector to the
curve $t\mapsto\varphi^\delta_{\rho(t,y)}(y)$ onto the unit normal
vector $\nu_{\rho(t)}(y)$ at that point, it holds that 
$$
V_{\Gamma_{\rho(t)}}=\big( \nu^\delta_{\rho(t,y)}(y) \big|
\nu_{\rho(t)}(y)\big)\dot\rho(t,y) \quad\text{ for }\quad
\varphi^\delta_{\rho(t,y)} (y)\in \Gamma(t) \text{ with } y\in
\Gamma_0. 
$$ 
To keep notation simple we often omit the time
variable. Therefore, \eqref{d4}--\eqref{d5} is equivalent to the
evolution equation for $\rho$,
\begin{equation}\label{rhoev}\begin{cases}
  \dot\rho=\bigl( \nu^\delta_\rho \big|\nu_\rho
  \bigr)^{-1}F\big(-\partial _{\nu_\rho}u_\rho \big)=:G(\rho),\quad
  t>0, &\\ \rho(0,\cdot)\equiv 0,&\end{cases}
\end{equation}
where $(u_\rho,\lambda_\rho)$ denotes the  solution of
\eqref{d1}--\eqref{d3} from Proposition~\ref{P1} in 
$\Omega_\rho$, the domain bounded by $\Gamma_\rho$, {and a given fixed
  $\mu\ge 0$ (note that $\rho\equiv 0$ gives the initial domain, that
  is, $\Omega_0=\Omega^0$)}. This is a
nonlinear, nonlocal equation for $\rho:[0,\infty)\times \Gamma_0\to
\mathbb{R}$. Notice that, for $\delta<<1$ and $|\rho|<<1$, the factor
of $\dot\rho$ in the expression for $V_{\Gamma_{\rho}}$ satisfies
$$
 \nu^\delta_\rho \cdot\nu_\rho \simeq \nu_0\cdot\nu_0=1.
$$
In order to obtain local well-posedness for \eqref{rhoev} using
maximal regularity techniques, its linearization at the initial
datum $\rho_0\equiv 0$ is computed. For this we first note 

\begin{prop}\label{P2}
There exists an open zero-neighborhood $\mathcal{O}$ in
$buc^{2+\alpha}(\Gamma_0)$ such that 
$$
 G\in C^2(\mathcal{O},buc^{1+\alpha}(\Gamma_0)).
$$
\end{prop}

\begin{proof}
This follows from \eqref{F} and the fact that the flow
$\varphi_r^\delta=\varphi^\delta(\cdot,r)$ is smooth with respect $r$
which implies that the solution  $(u_\rho,\lambda_\rho)$ of
\eqref{d1}--\eqref{d3} depends smoothly on $\rho$ since
$\Gamma_\rho=\varphi^\delta_\rho(\Gamma_0)$. We refer to \cite[Theorem
3.6]{PG2} for a more detailed and explicit calculation of this
derivative which yields the desired smoothness.
\end{proof}
We next verify that
$$
 DG(0)\in \mathcal{L} \bigl( buc^{2+\alpha}(\Gamma_0),
 buc^{1+\alpha}(\Gamma_0)\bigr)
$$
is the generator of an analytic semigroup. Observe that 
\begin{align}
 DG(0)h& =\left.\frac{d}{d\epsilon}\right|_{\epsilon=0} G(\epsilon h)\nonumber\\
&=-\frac{1}{\nu^\delta\cdot
  \nu_0}F'(-\partial_{\nu_0}u_0)\Big\{\left.\frac{d}{d\epsilon}\right|_{\epsilon=0}
 \partial_{\nu_{\epsilon h}}u_{\epsilon h}\Big\} \nonumber\\
&
 \quad-\frac{1}{(\nu^\delta\cdot \nu_0)^2}F(-\partial_{\nu_0}u_0)
 \Big\{\left.\frac{d}{d\epsilon}\right|_{\epsilon=0} 
 \bigl(\nu^\delta_{\epsilon h}\cdot \nu_{\epsilon h}\bigr)\Big\} \label{DG}
\end{align}
for $h\in buc^{2+\alpha}(\Gamma_0)$ since
$\nu^\delta_0=\nu^\delta$. It then remains to compute the derivatives
in the curly brackets. This is where the choice of coordinate system
from Lemma \ref{gentub} delivers its benefits yielding a particularly
insightful representation. 

\begin{lem}\label{normvar} It holds that
\begin{align*}
  \left.\frac{d}{d\epsilon}\right|_{\epsilon=0}  \nu_{\epsilon
   h}  = - \Big[ h\big( d_y\nu^\delta  [\tau_0] \big| \nu_0\big) +h'
  \big(\nu^\delta  \big|  \nu_0\big)\Big] \tau_0 
\end{align*}
{for $h\in buc^{2+\alpha}(\Gamma_0)$}, where $'$ denotes
differentiation with respect to arc length.
\end{lem}

\begin{proof}
It follows from $\tau_\rho\cdot\nu_\rho\equiv 0$ and
$\nu_\rho\cdot\nu_\rho\equiv 1$ that
$$ 
 \bigl( \frac{d}{d \epsilon }\tau_{\epsilon h}\bigr) \cdot
 \nu_{\epsilon h}=-\tau_{\epsilon h}\cdot \bigl( \frac{d}{d \epsilon }
 \nu_{\epsilon h}\bigr), \quad \bigl(\frac{d}{d \epsilon}\nu_{\epsilon
   h}\bigr)\cdot\nu_{\epsilon 
h}=0.  
$$
This implies that
\begin{equation}\label{x0}
 \left.\frac{d}{d\epsilon}\right|_{\epsilon=0}  \nu_{\epsilon
   h}  =-\big[\bigl( \left.\frac{d}{d\epsilon}\right|_{\epsilon=0}
 \tau_{\epsilon h} \bigr)\cdot \nu_0\big]\tau_0.
\end{equation}
It is therefore enough to compute
$\left.\frac{d}{d\epsilon}\right|_{\epsilon=0} \tau_{\epsilon
  h}$. Now, since $\gamma_\rho(s):=\varphi^\delta \Bigl(
\gamma_0(s),\rho\bigl(\gamma_0(s)\bigr)\Bigr)$ is a parametrization of
$\Gamma_\rho$ whenever $\gamma_0$ is an arc length parametrization of
$\Gamma_0$, one has from \eqref{flow} that
$$
\tilde\tau_{\epsilon h}:=d_y\varphi^\delta(\cdot,\epsilon
h)[\tau_0]+\epsilon  h' \nu^\delta_{\epsilon h}
$$
is a tangent vector to $\Gamma_{\epsilon h}$ and thus that
$\tau_{\epsilon h} ={\tilde\tau_{\epsilon h}}/{|\tilde\tau_{\epsilon
    h}|}$. The latter yields
\begin{align}
 \left.\frac{d}{d\epsilon}\right|_{\epsilon=0}\tau_{\epsilon h}
 &= \left.\frac{d}{d\epsilon}\right|_{\epsilon=0}
   \frac{\tilde\tau_{\epsilon h}}{|\tilde\tau_{\epsilon
   h}|}=\frac{1}{|\tilde\tau_0|}\left. 
 \frac{d}{d\epsilon}\right|_{\epsilon=0} \tilde\tau_{\epsilon h}-
 \frac{1}{|\tilde\tau_0|^3}\big[ \tilde\tau_0\cdot \bigl(\left.
 \frac{d}{d\epsilon}\right|_{\epsilon=0} \tilde\tau_{\epsilon h}
 \bigr)\big]\tilde\tau_0 \nonumber
\\&=\left. \frac{1}{|\tilde\tau_0|}\frac{d}{d\epsilon}\right|_{\epsilon=0}
 \tilde\tau_{\epsilon h}-\frac{1}{|\tilde\tau_0|} \big[ \tau_0\cdot \bigl(\left.
 \frac{d}{d\epsilon}\right|_{\epsilon=0} \tilde\tau_{\epsilon h}
 \bigr)\big]\tau_0. \label{x1}
\end{align}
Then, using \eqref{flow} again,
\begin{align}
\left.\frac{d}{d\epsilon}\right|_{\epsilon=0} \tilde\tau_{\epsilon
  h}&=d_y\big(\frac{d}{dr}\varphi^\delta(\cdot,0)\big)[\tau_0]h+h'\nu^\delta_0 \nonumber\\
  &=hd_y\nu^\delta _0 [\tau_0]+h' \nu^\delta _0.\label{x3}
\end{align}
Combining \eqref{x0}--\eqref{x3} and noticing
  $\tilde\tau_0=d_y\varphi^\delta(\cdot,0)[\tau_0] =\tau_0$ since
  $\varphi^\delta_0=\operatorname{id}_{\Gamma_0}$,  we get 
\begin{align*}
\left.\frac{d}{d\epsilon}\right|_{\epsilon=0} \nu_{\epsilon h}  & =- 
\big[\bigl( \left.\frac{d}{d\epsilon}\right|_{\epsilon=0}
 \tilde\tau_{\epsilon h}\bigr)\cdot \nu_0\big]\tau_0\\
& = - \big[ h\big(d_y\nu^\delta _0 [\tau_0] \big| \nu_0\big) +h'
  \big(\nu^\delta _0 \big| \nu_0\big)\big] \tau_0 
\end{align*}
and the claim follows from $\nu^\delta_0=\nu^\delta$.
\end{proof}

\begin{lem}\label{l32}
It holds that
\begin{align*}
\left.\frac{d}{d\epsilon}\right|_{\epsilon=0} \big(\nu^\delta_{\epsilon h}\cdot   \nu_{\epsilon
   h} \big) = \Big[ h \big(d_y\nu^\delta[\nu_0]\big| \nu_0\big)
	-  h' \big(\nu^\delta\big| \tau_0\big)\Big] \big(\nu^\delta\big|\nu_0\big) 
\end{align*}
for $h\in buc^{2+\alpha}(\Gamma_0)$.
\end{lem}

\begin{proof}
Owing to the previous lemma it only remains to compute
$$
 \left.\frac{d}{d\epsilon}\right|_{\epsilon=0} \nu^\delta_{\epsilon
   h}=\left.\frac{d}{d\epsilon}\right|_{\epsilon=0} \nu^\delta\big(\varphi^\delta(\cdot,\epsilon
   h)\big)= h d_y\nu^\delta[\nu^\delta] 
$$
and to note that 
$$
\nu^\delta-\big(\nu^\delta\big|\tau_0\big)\tau_0 =\big(\nu^\delta\big|\nu_0\big)\nu_0 .
$$
\end{proof}

\begin{rem}
Notice that, when $\delta\simeq 0$ one has that
$$
 d_y\nu^\delta[\nu^\delta]\simeq 0,\quad \nu^\delta\cdot\tau_0\simeq
 0,\ \text{ and }\ \nu^\delta_0\cdot\nu_0\simeq 1,
$$
uniformly on $\Gamma_0$. If $\delta$ can be set equal to zero, then
$$
 \left.\frac{d}{d\epsilon}\right|_{\epsilon=0} \nu^0_{\epsilon
   h}\cdot \nu_{\epsilon h}\equiv 0.
$$
\end{rem}

\begin{lem}\label{dtnvar}
Given $h\in buc^{2+\alpha}(\Gamma_0)$ and $\epsilon>0$ small,
let $u_{\epsilon h}$ solve the boundary value problem 
$$
\begin{cases}
 -\Delta u_{\epsilon h}=f&\text{in }\Omega_{\epsilon h}=\varphi^\delta
 _{\epsilon h}(\Omega_0),\\ u_{\epsilon h}=0&\text{on
 }\Gamma_{\epsilon h}=\varphi^\delta_{\epsilon h}(\Gamma_0),
\end{cases}
$$
for $f\in \operatorname{C}^\infty(\mathbb{R}^n)$. Then
$$
 \left.\frac{d}{d\epsilon}\right|_{\epsilon=0} \partial_{\nu_{\epsilon
 h}}u_{\epsilon h}=-DtN_{\Gamma_0}\bigl(
[\partial_{\nu^\delta}u_0]h\bigr)+ \bigl(
\nu_0^\mathsf{T}D^2u_0\nu^\delta \bigr) h, 
$$
where the Dirichlet-to-Neumann operator $DtN_\Gamma$ is the operator that yields the
normal derivative $\partial _{\nu_\Gamma}u_g$ (Neumann datum) of the
harmonic function $u_g$ with Dirichlet datum $g$, that is
$$
 DtN_\Gamma(g)=\partial _{\nu_\Gamma}u_g,
$$ 
and where $u_0$ is the solution corresponding to the boundary value
problem in $\Omega_0$.
\end{lem}
\begin{proof}
Assume first that  $h\le 0$, hence $\Omega_{\epsilon h}\subset
\Omega_0$. We look for $u_{\epsilon h}$ in the form 
$$
 u_{\epsilon h}=v_{\epsilon h}+u_0,
$$
Then $v_{\epsilon h}$ satisfies
$$
\begin{cases}
 -\Delta v_{\epsilon h}=0&\text{in }\Omega_{\epsilon h},\\ v_{\epsilon
   h}=-u_0\big |_{\Gamma_{\epsilon h}}&\text{on }\Gamma_{\epsilon h},
\end{cases}
$$
and
\begin{align*}
 \partial_{\nu_{\epsilon h}}u_{\epsilon h}=\partial_{\nu_{\epsilon
 h}}v_{\epsilon h}+\partial_{\nu_{\epsilon
  h}}u_0=-DtN_{\Gamma_{\epsilon h}}(u_0\big |_{\Gamma_{\epsilon
  h}})+  \nu_{\epsilon h}\cdot  \nabla u_0\big |_{\Gamma_{\epsilon
  h}} .
\end{align*}
It is known that the mapping
$$
  \mathcal{M}^{2+\alpha}\rightarrow
 \mathcal{L}\bigl( buc^{2+\alpha}(\Gamma),buc^{1+\alpha}(\Gamma)\bigr) ,\quad \Gamma\mapsto DtN_\Gamma,
$$
is a smooth local section of the corresponding bundle. Indicating with
a superscript $*$ the pull-back operation it follows that  (see
\cite[Section 3]{PG2} for more details)
\begin{align}\label{neumannvar}\notag
 \left.\frac{d}{d\epsilon}\right|_{\epsilon=0} 
 \partial_{\nu_{\epsilon h}}u_{\epsilon h}=&
  -\left.\frac{d}{d\epsilon}\right|_{\epsilon=0} \bigl( 
 \varphi^\delta _{\epsilon h}\bigr)^* DtN_{\Gamma_{\epsilon h}}
 \left(\bigl(\bigl(\varphi^\delta _{\epsilon h}\bigr)^{-1}\bigr)^* u_0\big
  |_{\Gamma_0}\right)-DtN_{\Gamma_0}\Bigl(
  \left.\frac{d}{d\epsilon}\right|_{\epsilon=0} \bigl(
  \varphi^\delta _{\epsilon h}\bigr)^*u_0\big |_{\Gamma_{\epsilon
  h}}\Bigr)\\ &+\left( \left.\frac{d}{d\epsilon}\right|_{\epsilon=0}
                \nu_{\epsilon h}\right) \cdot \nabla u_0\big
                |_{\Gamma_0} +   \nu_0\cdot  
                \left(\left.\frac{d}{d\epsilon}\right|_{\epsilon=0} \bigl(
  \varphi^\delta _{\epsilon h}\bigr)^*\nabla u_0\big
                |_{\Gamma_{\epsilon h}}\right)\\\notag
 =&-DtN_{\Gamma_0}\Bigl( [\partial_{\nu^\delta}u_0]h\Bigr)+
    \Bigl(\sum_{j,k=1}^2\nu_0^j(\nu^\delta )^k \frac{\partial^2}{\partial
    x^j \partial x^k}u_0\Bigr)h,
\end{align}
since the first term after the first equality sign vanishes in view of
the homogeneous Dirichlet condition satisfied by $u_0$ and the third 
in view of Lemma \ref{normvar} and of the boundary condition again. It
remains to show that the result remains valid for any $h$ without the
restriction that $h\leq 0$. To that end, define
$\Gamma_{r_0}=\varphi^\delta(\Gamma_0,r_0)$ for $r_0>0$ small enough
and replace the solution $u_0$ by the solution $u_{r_0}$ in the above
argument. At the end of the calculation, formula \eqref{neumannvar} is
obtained with all terms after the first equality sign
non-vanishing. Letting $r_0$ tend to zero makes them vanish and delivers the
claim. For more details we refer to the proof of \cite[Theorem 3.7]{PG2}.
\end{proof}

\begin{lem}\label{l35}
Given $h\in buc^{2+\alpha}(\Gamma_0)$ it holds that
$$
 \left.\frac{d}{d\epsilon}\right|_{\epsilon=0} \partial_{\nu_{\epsilon
 h}}u_{\epsilon h}=-DtN_{\Gamma_0}\Bigl(
 [\partial_{\nu^\delta}u_0]h\Bigr)+\Bigl(\nu_0^\mathsf{T}
 D^2u_0\nu^\delta\Bigr)h+(\partial_{\nu_0}u^{(1)}_0)
 \left.\frac{d}{d\epsilon}\right|_{\epsilon=0} \lambda_{\epsilon h},
$$
where
\begin{equation}\label{lambda}
 \lambda_{\epsilon h}=\frac{{\mathcal{V}}}{\int _{\Omega_{\epsilon
 h}}u^{(1)}_{\epsilon h}\, dx}-{\mu}\frac{\int _{\Omega_{\epsilon
 h}}u^{(x^1)}_{\epsilon h}\, dx}{\int _{\Omega_{\epsilon h}}u^{(1)}_{\epsilon h}\, dx}.
\end{equation}
\end{lem}

\begin{proof}
Recall from Proposition~\ref{P1} that
$$
 u_{\epsilon h}=\mu u^{(x^1)}_{\epsilon h}+\lambda_{\epsilon h}u^{(1)}_{\epsilon h}  
$$
with $\lambda_{\epsilon h}$ given as in the statement. Lemma
\ref{dtnvar} implies that
\begin{align*}
\left.\frac{d}{d\epsilon}\right|_{\epsilon=0} \partial_{\nu_{\epsilon
 h}} u^{(x^1)}_{\epsilon
  h}&=-DtN_{\Gamma_0}\Bigl( [\partial_{\nu^\delta
      }u_0^{(x^1)}]h\Bigr)+\Bigl(\nu_0^\mathsf{T}D^2u_0^{(x^1)}\nu^\delta\Bigr)
  h ,\\ 
	\left.\frac{d}{d\epsilon}\right|_{\epsilon=0} \partial_{\nu_{\epsilon
 h}} u^{(1)}_{\epsilon h} &=-DtN_{\Gamma_0}\Bigl( [\partial_{\nu^\delta
      }u_0^{(1)}]h\Bigr)+
  \Bigl(\nu_0^\mathsf{T}D^2u_0^{(1)}\nu^\delta\Bigr) h.
\end{align*}
It follows that
$$
 \left.\frac{d}{d\epsilon}\right|_{\epsilon=0}\partial_{\nu_{\epsilon
 h}}u_{\epsilon h}=-DtN_{\Gamma_0}\Bigl(
 [\partial_{\nu^\delta}u_0]h\Bigr)+\Bigl(\nu_0^\mathsf{T}
 D^2u_0\nu^\delta\Bigr)h+(\partial_{\nu_0}u^{(1)}_0)
 \left.\frac{d}{d\epsilon}\right|_{\epsilon=0} \lambda_{\epsilon h}.
$$
\end{proof}
Combining the results of \eqref{DG}, Lemma~\ref{l32}, and
Lemma~\ref{l35} the linearization of $G$ at zero is seen to be given
by the expression
\begin{align*}
 DG(0)h=& -\frac{1}{\nu^\delta\cdot \nu_0}
 F'(-\partial_{\nu_0}u_0)\Big\{- DtN_{\Gamma_0}\Bigl(
  [\partial_{\nu^\delta}u_0]h\Bigr)+\Bigl(\nu_0^\mathsf{T} 
 D^2u_0\nu^\delta\Bigr)h +(\partial_{\nu_0}u^{(1)}_0)
 \left.\frac{d}{d\epsilon}\right|_{\epsilon=0} \lambda_{\epsilon h}\Bigr\}\\
& - \frac{1}{(\nu^\delta\cdot \nu_0)}F(-\partial_{\nu_0}u_0)\Bigl\{
 h \big(d_y\nu^\delta[\nu_0]\big| \nu_0\big)
	-  h' \big(\nu^\delta\big| \tau_0\big)\Bigr\}\\
 =:&\: I+II+III+IV+V,
\end{align*}
for $h\in buc^{2+\alpha}(\Gamma_0)$, where $(u_0,\lambda_0)$ is
the solution to \eqref{d1}--\eqref{d3} in $\Omega_0$ from
Proposition~\ref{P1}. From this formula we derive the following
generation result.

\begin{thm}\label{T1}
Suppose \eqref{F} and let $\Gamma_0\in \mathcal{M}^{2+\alpha}$. 
Then
$$
 -DG(0)\in \mathcal{H} \bigl( buc^{2+\alpha}(\Gamma_0),
 buc^{1+\alpha}(\Gamma_0)\bigr) 
$$
{for  $\mu\in[0,\mu_0)$, where $\mu_0>0$ is small enough} depending on
$\Omega_0$ and ${\mathcal{V}}$. In other words, $DG(0)$ generates
an analytic {$\operatorname{C}^0$}-semigroup on
$buc^{1+\alpha}(\Gamma_0)$ for $\mu <\mu_0$ and its domain of definition
coincides with $buc^{2+\alpha}(\Gamma_0)$.  There is $\mu_1\ge \mu_0$ such that $$
 DG(0)\in \mathcal{H} \bigl( buc^{2+\alpha}(\Gamma_0),
 buc^{1+\alpha}(\Gamma_0)\bigr) 
$$
for $\mu\ge\mu_1$, which makes the evolution equation linearly ill-posed.
\end{thm}

\begin{proof}
First observe that for $\partial_{\nu^\delta }u_0$ we have 
\begin{align*}
 -\Delta {\partial_{\nu^\delta }u_0}=&-\sum_{j,k=1}^2\partial _j^2
  \bigl( (\nu^\delta )^k \partial _ku_0\bigr) =-\sum_{j,k=1}^2
 \Big\{ \bigl(\partial_j^2(\nu^\delta )^k\bigr)\partial_k
  u_0+2\partial_j(\nu^\delta )^k \partial_j \partial_ku_0+(\nu^\delta )^k
  \partial_j^2 \partial_k u_0\Bigr\}\\
 =& -\Delta\nu^\delta \cdot\nabla u_0-2D\nu^\delta
    :D^2u_0-\nu^\delta \cdot\nabla\Delta u_0\\
 =&-\Delta\nu^\delta \cdot\nabla u_0-2D\nu^\delta
    :D^2u_0+\mu (\nu^\delta )^1\in buc^{\alpha}(\Omega_0)
\end{align*}
by definition of $u_0\in  buc^{2+\alpha}(\Omega_0)$ and
\begin{align}\label{nunu}\notag
 \partial_{\nu_0}\bigl( \partial_{\nu^\delta }u_0\bigr)&=\sum_{k=1}^2 \nu_0^k \partial_k
 \Bigl( \bigl[ (\nu^\delta |\nu_0)\nu_0+(\nu^\delta |\tau_0)\tau_0
 \bigr]\cdot\nabla u_0\Bigr)\\\notag &= \bigl(\partial_{\nu_0}
 (\nu^\delta |\nu_0)\bigr)\partial_{\nu_0}u_0+(\nu^\delta |\nu_0)
 \partial _{\nu_0\nu_0}u_0\\
 &= \bigl( \partial_{\nu_0}\nu^\delta\big |
    \nu_0\bigr)\partial_{\nu_0}u_0- (\nu^\delta |\nu_0)f\in buc^{1+\alpha}(\Gamma_0) 
\end{align}
for $f(x):=\mu x^1+\lambda$, where we used for the second equality
that $\partial_{\tau_0} u_0=0$ owing to the Dirichlet boundary
condition and for the third equality the fact that
$\partial_{\nu_0}\nu_0= 0$ along with $$\partial
  _{\nu_0\nu_0}u_0= \partial _{\nu_0\nu_0}u_0+\partial
  _{\tau_0\tau_0}u_0 =\Delta u_0=-f \ \text{ on } \Gamma_0
  .$$
Consequently, classical theory of boundary value problems implies that
$$
 \partial_{\nu^\delta }u_0\in buc^{2+\alpha}(\Gamma_0).
$$
{From this it follows that
$$
 \left[h\mapsto \frac{1}{\nu^\delta\cdot \nu_0}
 F'(-\partial_{\nu_0}u_0) DtN_{\Gamma_0}\Bigl(
  [\partial_{\nu^\delta}u_0]h\Bigr)\right] \in \mathcal{L} \bigl( buc^{2+\alpha}(\Gamma_0),
 buc^{1+\alpha}(\Gamma_0)\bigr).
$$}
Next notice that the map that associates {with} a curve $\Gamma$ the
  corresponding $\lambda$ from Proposition~\ref{P1} is well-defined
in a neighborhood of $\Gamma_0$ in $\mathcal{M}^{2+\alpha}$.
Consequently, its tangential map
$$
  buc^{2+\alpha}(\Gamma_0)\longrightarrow
 \mathbb{R},\quad h\mapsto \left.\frac{d}{d\epsilon}\right|_{\epsilon=0}
 \lambda_{\epsilon h}
$$
is a linear, rank 1, and, hence, compact operator. Now term $I$ is the
most important one and defines an elliptic pseudodifferential operator
of order~1 whenever 
$$
 -\partial _{\nu^\delta }u_0>0,\: F'(-\partial
 _{\nu_0}u_0)>0,\ \text{ and }\ \nu^\delta \cdot \nu_0 >0
 \quad \text{on }\ \Gamma_0. 
$$
The second condition is satisfied by
assumption~\eqref{F}. The first  holds true if
$-\partial_{\nu_0}u_0>0$ on $\Gamma_0$, which 
follows from Proposition~\ref{P1}. 
Indeed, if $\delta$ is small enough, then $\nu^\delta \cdot
\nu_0\simeq 1$ (which also guarantees the validity of the third
condition), and therefore we have that
$$
 -\partial_{\nu^\delta }u_0\simeq -\partial _{\nu_0}u_0>0,
$$
thanks to the uniform convergence in $buc^{1+\alpha}$ in
\eqref{normreg}. This implies that $I$ is in fact the generator of an
analytic {$\operatorname{C}^0$}-semigroup on 
$buc^{1+\alpha}(\Gamma_0)$ with domain $buc^{2+\alpha}(\Gamma_0)$. 
A complete argument would require a standard localization argument
based on the smoothness of the coefficients and a symbol analysis of
the corresponding frozen coefficients operator (see e.g. \cite{ES,EG}
for more details). In this case, the latter has the explicit form
$a_0|\xi|$ because of the particularly insightful form of the main
term of $DG(0)$, which, it is reminded, is itself a consequence of the
use of coordinates constructed by means of the flow
$\varphi^\delta$. As for the remaining terms, they are lower order
perturbations (as multiplication operators), like
$II-IV$ thanks to regularity of the coefficients (using also
\eqref{nunu}), or a small perturbation, as for $V$, thanks to
$$
 \nu^\delta \cdot\tau_0 \simeq 0.
$$
Recall for the latter that the set of analytic generators is open in
the natural operator topology of
$\mathcal{L}\bigl(buc^{2+\alpha}(\Gamma_0),buc^{1+\alpha}(\Gamma_0)\bigr)$.
The first assertion is therefore proved. As for the second, we note
that Proposition~\ref{P1} yields  $\mu_1\ge \mu_0$ such
$-\partial_{\nu^\delta}u_0<0$ on $\Gamma_0$ {for $\mu\ge \mu_1$},
and the claim follows by 
approximation as above.
\end{proof}

Existence results based on maximal regularity  can now be
used to derive the following

\begin{thm}\label{T2}
Given any $\Gamma_0\in buc^{2+\alpha}$, there is a $\mu_0>0$
(depending on {$\Gamma_0$} and ${\mathcal{V}}$) such that, {for all
  $\mu\in [0,\mu_0)$, system} \eqref{d1}--\eqref{d5}
is well-posed on {some maximal interval $[0,T_0)$}. The
solution $(u,\Gamma)$ satisfies 
$$
 \Gamma(t)=\Gamma _{\rho(t,\cdot)},\quad t\in [0,T_0),
$$
with
$$
 \rho\in \operatorname{C}\bigl( [0,T_0),
 buc^{2+\alpha}(\Gamma_0)\bigr)\cap \operatorname{C}^1 
 \bigl( [0,T_0), buc^{1+\alpha}(\Gamma_0)\bigr),
$$
and
$$
 u(t,\cdot)\in
 buc^{2+\alpha}\bigl( \Omega(t)\bigr),\quad t\in [0,T_0),
$$
for $\Omega(t)=\Omega_{\rho(t,\cdot)}$.
There is $\mu_1\ge \mu_0$ such that the system \eqref{d1}--\eqref{d5}
is linearly ill-posed for $\mu\ge \mu_1$.
\end{thm}

\begin{proof}
Since
$\mathcal{H}\bigl(buc^{2+\alpha}(\Gamma_0),buc^{1+\alpha}(\Gamma_0)\bigr)$
is open in
$\mathcal{L}\bigl(buc^{2+\alpha}(\Gamma_0),buc^{1+\alpha}(\Gamma_0)\bigr)$,
it follows from Theorem~\ref{T1} and Proposition~\ref{P2} that we may
assume with loss of generality that  
$$
-DG(\rho)\in
\mathcal{H}\bigl(buc^{2+\alpha}(\Gamma_0),buc^{1+\alpha}(\Gamma_0)\bigr)\
,\quad \rho\in\mathcal{O},
$$
and the existence assertion follows e.g. from
\cite[Theorem 8.4.1]{Lunardi} (or from \cite[Thorem 5.6]{PG2}) and the
fact that the little H\"older spaces are stable under continuous interpolation.
\end{proof}

\begin{rem}
The above result confirms and quantifies the physical intuition that a
solution exists only for small incline angle and ceases to exist for
larger angles. Obviously, the determining critical size of the angle
depends on $\Gamma_0$ and ${\mathcal{V}}$, i.e. the shape of the
initial wetted region and the mass of liquid since these determine the
sign of $\partial_{\nu_0}u_0$.
\end{rem}

\section{Stability Analysis for Translating Circular Solutions}

We  assume throughout the following that \eqref{affine} is satisfied, that is,
$$
F(q)=a q-b,\quad q\in \mathbb R,
$$
for some $a,b>0$.
Recall that if $R_0$ and $v_0$ are as in \eqref{R0}, then $u_0$
defined in \eqref{ui} solves \eqref{d1b}-\eqref{d4b} on the disk
$\Omega_*=R_0\mathbb B$ (and is positive provided $\mu$ is small). We
now investigate the asymptotic stability of this translating circular
solution for small inclines, i.e. for small $\mu\ge 0$.

\subsection{Reformulation}

We rewrite problem \eqref{d1}-\eqref{d5} by introducing  the translations
$$
\tilde u (t,x):=u(t,x+tv_0e_1),\quad x\in \tilde\Omega(t):=\Omega(t)-tv_0e_1.
$$ 
Substituting this into \eqref{d1}--\eqref{d5} and
dropping again the tildes everywhere for ease of notation, it is
readily seen that \eqref{d1}--\eqref{d5} is equivalent to
 \begin{alignat}{2}\label{d1c}
  -\Delta u&=\mu x^1+\lambda&&\text{in }\Omega(t),\\
 \label{d2c}
  u&=0&&\text{on }\Gamma(t),\\\label{d3c}
  \int _{\Omega(t)}u\, dx&=\mathcal{V},&&\\\label{d4c}
 V&=F( -\partial _\nu u)-v_0 e_1\cdot \nu\qquad&&\text{on }\Gamma(t),\\
\Omega(0)&=\Omega^0,  \label{d5c}
 \end{alignat}
and $(u_0,\Omega_*)$ is a stationary solution to
\eqref{d1c}-\eqref{d5c}. As in the previous section, we can consider
this problem as a single equation for the geometry.
Let $\nu_0(y)=y/R_0$ denote the normal at $y\in\Gamma_0:=R_0\mathbb
S^1$. Note that in this case, since $\nu_0$ is smooth, we can take
$\nu^\delta=\nu_0$ in Section~\ref{Sec2}, and the flow in \eqref{flow}
is simply given by $\varphi^\delta(y,r)=(1+r/R_0) y$. Thus, the
evolution of $\Gamma(t)$ is described by the evolution of the function 
$\rho(t,\cdot):\Gamma_0\to \mathbb{R}$  through
$$
 \Gamma(t)=\big\{ y+\rho(t,y)\nu_0(y)\,\big |\, y\in
 \Gamma_0\big\}=:
 \Gamma_{\rho(t)},
$$
and $\rho$ is governed by
\begin{equation}\label{rhoevd}\begin{cases}
  \dot\rho=\bigl( \nu_0 \big|\nu_\rho
  \bigr)^{-1}\big(F\big(-\partial _{\nu_\rho}u_\rho \big)-v_0e_1\cdot \nu_\rho\big)=:H(\rho),\qquad
  t>0,\quad y\in \Gamma_0, &\\ \rho(0,\cdot)= \rho_0,&\end{cases}
\end{equation}
provided that $\partial\Omega^0=\Gamma_{\rho_0}$, which is possible
for $\Omega^0$ sufficiently close to
$\Omega_*$ (i.e. for $\rho_0$
is small enough). Here, $\nu_\rho=\nu_{\Gamma_\rho}$, and $u_\rho$
with corresponding $\lambda_\rho$ denotes the  solution of 
\eqref{d1}--\eqref{d3} from Proposition~\ref{P1} in 
$\Omega_\rho$, the domain enclosed by $\Gamma_\rho$. Clearly,
$\Omega_0=\Omega_*=R_0\mathbb B$. Note that
\begin{equation}\label{H}
 H\in C^2(\mathcal{O},buc^{1+\alpha}(\Gamma_0)),\qquad H(0)=0,
\end{equation}
according to Proposition~\ref{P2} with $\mathcal{O}$ being an open
zero-neighborhood in $buc^{2+\alpha}(\Gamma_0)$, 
 and $\rho=0$ gives rise to the stationary solution
 $(u_0,\Omega_0)$. We shall prove that $\rho=0$ is locally
 asymptotically stable for equation \eqref{rhoevd} by using the
 principle of linearized stability.

\subsection{Linearization}

We now express the Fr\'echet derivative $DH(0)$ in terms of Fourier
expansions. For this we use polar coordinates  $(r,\theta)$ and
Cartesian coordinates $x$ interchangeably and observe that, if $h\in
buc^{2+\alpha}(\Gamma_0)$, then the form of $F$ implies that
$$
DH(0) h= \left.\frac{d}{d\epsilon}\right|_{\epsilon=0} H(\epsilon h)=
-a  \left.\frac{d}{d\epsilon}\right|_{\epsilon=0}
\big(\partial_{\nu_{\epsilon h}} u_{\epsilon h}\big)-v_0
\left.\frac{d}{d\epsilon}\right|_{\epsilon=0} \left( e_1\cdot
  \nu_{\epsilon h}\right).
$$
Therefore, Lemma~\ref{normvar} and Lemma~\ref{l35} entail that
\begin{equation}\label{DH}
\begin{split}
DH(0) h &= a DtN_{\Gamma_0}\bigl(
 [\partial_{\nu_0}u_0]h\bigr)-a\bigl(\nu_0^\mathsf{T}
 D^2u_0\nu_0\bigr)h-a\partial_{\nu_0}u^{(1)}_0
 \left.\frac{d}{d\epsilon}\right|_{\epsilon=0} \lambda_{\epsilon h} \\
&\quad 
+v_0 \big[ h\big( d_y\nu_0  [\tau_0] \big| \nu_0\big) +h'\big] \tau_0\cdot e_1
\end{split}
\end{equation}
with $\tau_0$ denoting the unit tangent vector to
$\Gamma_0$. We compute the different terms separately. Let $h\in
buc^{2+\alpha}(\Gamma_0)$ be given with representation
$$
h(\theta)=\sum_{n\in\Z} \hat h_n e^{in\theta},\quad \theta\in [0,2\pi].
$$
Then
$$
w(r,\theta)=\sum_{n\in\Z} \hat h_n \left(\frac{r}{R_0}\right)^{\vert n\vert} e^{in\theta}
$$
solves
$$
-\Delta w=0\quad \text{in }\ \Omega_0,\qquad w=h \quad \text{on }\ \Gamma_0,
$$
and thus
$$
DtN_{\Gamma_0}(h)=\partial_r w(R_0,\theta)=\frac{1}{R_0}\sum_{n\in\Z}\hat h_n \vert n\vert e^{in\theta} .
$$
Next, since \eqref{ui} implies
$$
\partial_{\nu_0}u_0(R_0,\theta)=\partial_r u_0(R_0,\theta)=-\frac{\mu
  R_0^2}{8}\left(e^{i\theta}+e^{-i\theta}\right)
-\frac{4\mathcal{V}}{\pi R_0^3}, 
$$
it follows that
$$
\partial_{\nu_0}u_0 h=-\sum_{n\in\Z} \left[\frac{\mu R_0^2}{8}\left(
    \hat h_{n-1}+\hat h_{n+1}\right) +\frac{4\mathcal{V}}{\pi
    R_0^3}\hat h_n\right] e^{in\theta}. 
$$
In summary, we derive that
\begin{equation}\label{r7}
DtN_{\Gamma_0}\big(\partial_{\nu_0}u_0 h\big)= -\sum_{n\in\Z}
\left[\frac{\mu R_0}{8}\left( \hat h_{n-1}+\hat h_{n+1}\right)
  +\frac{4\mathcal{V}}{\pi R_0^4}\hat h_n\right] \vert n\vert
e^{in\theta}. 
\end{equation}
Similarly, we have from \eqref{ui}
$$
\bigl(\nu_0^\mathsf{T}
 D^2u_0\nu_0\bigr)h =\left(\partial_r^2 u_0\right) h=
 -\left(\frac{3\mu R_0}{8} \left(e^{i\theta}+e^{-i\theta}\right)
   +\frac{4\mathcal{V}}{\pi R_0^4}\right) h, 
$$
hence
\begin{equation}\label{r8}
\bigl(\nu_0^\mathsf{T}
 D^2u_0\nu_0\bigr)h = -\sum_{n\in\Z} \left[\frac{3\mu R_0}{8}\left(
     \hat h_{n-1}+\hat h_{n+1}\right) +\frac{4\mathcal{V}}{\pi
     R_0^4}\hat h_n\right]  e^{in\theta}. 
\end{equation}
Before computing the third term of the right-hand side of \eqref{DH},
we focus on the last term. Clearly, $\big( d_y\nu_0  [\tau_0] \big| \nu_0\big)=0$ so that
$$
v_0 \big[ h\big( d_y\nu_0  [\tau_0] \big| \nu_0\big) +h'\big]
\tau_0\cdot e_1=-v_0\sin\theta h', 
$$
where the derivative with respect to arc length is
$h'=\frac{1}{R_0}\frac{d}{d\theta} h$. Consequently, since $v_0=\mu a
R_0^2/4$ by \eqref{R0}, 
\begin{equation}\label{r9}
v_0 \big[ h\big( d_y\nu_0  [\tau_0] \big| \nu_0\big) +h'\big]
\tau_0\cdot e_1= -\frac{\mu a R_0}{8}\sum_{n\in\Z} \left[(n-1)\hat
  h_{n-1}-(n+1)\hat h_{n+1}\right]  e^{in\theta}. 
\end{equation}
As for the third term on the right-hand side of \eqref{DH}
recall, formula \eqref{lambda} for $\lambda_{\epsilon h}$. Then 
\begin{equation*}
 \left.\frac{d}{d\epsilon}\right|_{\epsilon=0} \lambda_{\epsilon h} =
 -\frac{ \mathcal{V}-\mu\int _{\Omega_{0}}u^{(x^1)}_{0}\,
   dx}{\left(\int _{\Omega_{0}}u^{(1)}_{0}\, dx\right)^2}\
 \left.\frac{d}{d\epsilon}\right|_{\epsilon=0} \left(\int
   _{\Omega_{\epsilon h}}u^{(1)}_{\epsilon h}\, dx\right) -
 \frac{\mu}{\int _{\Omega_{0}}u^{(1)}_{0}\, dx}\ 
\left.\frac{d}{d\epsilon}\right|_{\epsilon=0} \left(
\int _{\Omega_{\epsilon
 h}}u^{(x^1)}_{\epsilon h}\, dx\right),
\end{equation*}
and thus, on using 
$$
u^{(x^1)}_{0}(r,\theta)=\frac{r}{8} (R_0^2-r^2)\cos \theta,\qquad
u^{(1)}_{0}(r,\theta)= \frac{1}{4}(R_0^2-r^2),\qquad
\lambda_0=\frac{8\mathcal{V}}{\pi R_0^4}, 
$$ 
according to \eqref{ui}, we deduce
\begin{equation*}
 \left.\frac{d}{d\epsilon}\right|_{\epsilon=0} \lambda_{\epsilon h} =
 -\frac{ 64\mathcal{V}}{\pi^2R_0^8}\,
 \left.\frac{d}{d\epsilon}\right|_{\epsilon=0} \left(\int
   _{\Omega_{\epsilon h}}u^{(1)}_{\epsilon h}\, dx\right) -
 \frac{8\mu}{\pi R_0^4}\, 
\left.\frac{d}{d\epsilon}\right|_{\epsilon=0} \left(
\int _{\Omega_{\epsilon
 h}}u^{(x^1)}_{\epsilon h}\, dx\right).
\end{equation*}
We then invoke the transport theorem to get
$$
\left.\frac{d}{d\epsilon}\right|_{\epsilon=0} \left(
\int _{\Omega_{\epsilon
 h}}u^{(f)}_{\epsilon h}\, dx\right) =
\int_{\Omega_0}\left.\frac{d}{d\epsilon}\right|_{\epsilon=0}
u_{\epsilon h}^{(f)}\, d x + \int_{\Gamma_0} u_0^{(f)} {h}\,d
\sigma_{\Gamma_0}=\int_{\Omega_0}\left.\frac{d}{d\epsilon}\right|_{\epsilon=0}
u_{\epsilon h}^{(f)}\, d x
$$
since $u_{0}^{(f)}$ vanishes on $\Gamma_0$. Consequently,
\begin{equation}\label{lambda2}
 \left.\frac{d}{d\epsilon}\right|_{\epsilon=0} \lambda_{\epsilon h} =
 -\frac{ 64\mathcal{V}}{\pi^2R_0^8}\, \int
 _{\Omega_{0}}\left(\left.\frac{d}{d\epsilon}\right|_{\epsilon=0}
   u^{(1)}_{\epsilon h}\right)\, dx - \frac{8\mu}{\pi R_0^4}\, 
\int _{\Omega_{0}} \left(\left.\frac{d}{d\epsilon}\right|_{\epsilon=0}
  u^{(x^1)}_{\epsilon h}\right)\, dx. 
\end{equation}
It remains to compute the derivatives for which we proceed as in
\cite{PG1}. Given a smooth function $f$ on $\bar\Omega_{\epsilon h}$,
let 
$$
u_{\epsilon h}^{(f)}(x)=\int_{\Omega_{\epsilon h}}{K}_{\epsilon
  h}(x,\bar x) f(\bar x)\,d \bar x,\quad x\in \bar\Omega_{\epsilon
  h}, 
$$
be a representation of the solution to the Dirichlet problem
$$
-\Delta u=f\quad\text{in } \ \Omega_{\epsilon h},\qquad u=0
\quad\text{on } \ \Gamma_{\epsilon h}, 
$$
with Green's function ${K}_{\epsilon h}$ on $\Omega_{\epsilon h}$. 
In particular, for the circle $\Omega_0=R_0\mathbb B$ we have
$$
{K}_0(r,\theta,\bar
r,\bar\theta)=\frac{1}{4\pi}\log\left(\frac{R_0^2\left(r^2+\bar r^2-2r
      \bar r\cos(\theta-\bar\theta)\right)}{r^2\bar r^2+R_0^4-2R_0^2 r
    \bar r\cos(\theta-\bar\theta)}\right). 
$$
Then, noticing that ${K}_{\epsilon h}$ vanishes on $\Omega_{\epsilon h}\times \Gamma_{\epsilon h}$, we obtain
\begin{equation*}
\begin{split}
\left.\frac{d}{d\epsilon}\right|_{\epsilon=0} u_{\epsilon h}^{(f)}(x) &=\int_{\Gamma_0} {K}_0(x,\bar x) f(\bar x)\, d\sigma_{\Gamma_0}(\bar x)+\int_{\Omega_0} \left.\frac{d}{d\epsilon}\right|_{\epsilon=0} {K}_{\epsilon h}(x,\bar x) f(\bar x)\, d \bar x\\
&= \int_{\Omega_0} \left.\frac{d}{d\epsilon}\right|_{\epsilon=0}
{K}_{\epsilon h}(x,\bar x) f(\bar x)\, d \bar x. 
\end{split}
\end{equation*}
Next observe \cite{Had, Schi} that
$$
\left.\frac{d}{d\epsilon}\right|_{\epsilon=0} {K}_{\epsilon h}(x,\bar x)
= {\frac{1}{R_0}}\int_0^{2\pi}\partial_r {K}_0(x,R_0,\phi)\partial_r
{K}_0(\bar x, R_0,\phi) h(\phi)\, d \phi
$$
and that
$$
(r,\theta)\mapsto\int_0^{2\pi}\partial_r {K}_0(r,\theta,R_0,\phi)g(\phi)\, d \phi
$$
is the unique harmonic function in $\Omega_0$ with boundary value $g$
on $\Gamma_{0}$. Therefore, 
$$
\int_0^{2\pi} \partial_r {K}_0(\bar r,\bar\theta,R_0,\phi)
\cos\bar\theta \, d\bar\theta =\frac{\bar r}{R_0} \cos\phi
$$
since ${K}_0(\bar r,\bar\theta,\bar r,\phi) $ is symmetric with
respect to the angular (and radial) variables. Using this, we obtain 
\begin{equation*}
\begin{split}
\left.\frac{d}{d\epsilon}\right|_{\epsilon=0} u_{\epsilon
  h}^{(x^1)}(r,\theta) &= \int_{\Omega_0}
\left.\frac{d}{d\epsilon}\right|_{\epsilon=0} {K}_{\epsilon h}(x,\bar
x) f(\bar x)\, d \bar x\\ 
&= {\frac{1}{R_0}}\int_0^{2\pi}\int_0^{R_0} \partial_r
{K}_0(r,\theta,R_0,\phi) 
\int_0^{2\pi} \partial_r {K}_0(\bar r,\bar\theta,R_0,\phi) \bar r
\cos\bar\theta \, d\bar\theta \, h(\phi) \bar r\, d\bar r\, d\phi\\ 
&= {\frac{1}{R_0^2}}\int_0^{2\pi}\int_0^{R_0} \partial_r
{K}_0(r,\theta,R_0,\phi) \bar r^3 
 h(\phi) \cos\phi\,  d\bar r\, d\phi\\
&= {\frac{R_0^2}{4}}\int_0^{2\pi}  \partial_r {K}_0(r,\theta,R_0,\phi)
 h(\phi) \cos\phi\,   d\phi\,.
\end{split}
\end{equation*}
Integrating this  yields similarly
\begin{equation*}
\begin{split}
\int_{\Omega_0}\left.\frac{d}{d\epsilon}\right|_{\epsilon=0}
u_{\epsilon h}^{(x^1)}\, d x &=
{\frac{R_0^2}{4}}\int_0^{2\pi}\int_0^{R_0}\int_0^{2\pi}  \partial_r
{K}_0(r,\theta,R_0,\phi)\,d\theta\,
   r\,d r\, h(\phi) \cos\phi\,   d\phi \\
&={\frac{R_0^2}{4}}\int_0^{2\pi}\int_0^{R_0} r\,d r\,
 h(\phi) \cos\phi\,   d\phi \, ={\frac{R_0^4}{8}}\int_0^{2\pi} 
 h(\phi) \cos\phi\,   d\phi 
\end{split}
\end{equation*}
and finally
\begin{equation}\label{rB}
\int_{\Omega_0}\left.\frac{d}{d\epsilon}\right|_{\epsilon=0}
u_{\epsilon h}^{(x^1)}\, d x = {\frac{\pi R_0^4}{8}}\left(\hat
  h_{-1}+\hat h_1\right).
\end{equation}
Exactly the same way one computes
\begin{equation}\label{rA}
\int_{\Omega_0}\left.\frac{d}{d\epsilon}\right|_{\epsilon=0}
u_{\epsilon h}^{(1)}\, d x = {\frac{\pi R_0^3}{2}} \hat h_{0}
.
\end{equation}
It now follows from \eqref{lambda2}-\eqref{rA} that
\begin{equation}\label{r10}
 \left.\frac{d}{d\epsilon}\right|_{\epsilon=0} \lambda_{\epsilon h} =
 -\frac{ 32\mathcal{V}}{\pi R_0^5}\, \hat h_0- \mu \left(\hat
   h_{-1}+\hat h_1\right). 
\end{equation}
Consequently, gathering \eqref{DH}-\eqref{r9} and \eqref{r10} and
using $\partial_{\nu_0}u^{(1)}_0=-R_0/2$, we obtain the Fourier
expansion of $DH(0)$ in the form
\begin{equation}\label{r12}
\begin{split}
DH(0) h& =-\frac{12\mathcal{V}a}{\pi R_0^4}\, \hat h_0 -
\frac{4\mathcal{V}a}{\pi R_0^4}\sum_{n\in\Z^*}\big(\vert n\vert
-1\big)\,\hat h_n \, e^{in\theta}\\
&\quad
- \mu\frac{aR_0}{8}\sum_{n\in\Z^*}\left[\big(\vert n\vert +n-4\big) \,
  \hat h_{n-1} + \big(\vert n\vert -n-4\big) \, \hat h_{n+1}\right]\,
e^{in\theta}
\end{split}
\end{equation}
for $h\in buc^{2+\alpha}(\Gamma_0)$ with
$
h(\theta)=\sum_{n\in\Z} \hat h_n e^{in\theta}$, where
$\Z^*:=\Z\setminus\{0\}$. Note that the matrix is tridiagonal if
$\mu>0$. Since $DH(0)$ has a compact resolvent, its spectrum is
discrete and contains only eigenvalues. More information
is found in the following lemma.

\begin{lem}\label{L51}
Let $\mu\ge 0$ be sufficiently small. The kernel of $DH(0)$ is
two-dimensional and spanned by $\{ e^{-i\theta},
e^{i\theta}\}$. Moreover, there is $\omega>0$ independent of $\mu$
such that $\sigma(DH(0))\subset [\mathrm{Re} z\le -\omega]\cup
\{0\}$.
\end{lem}

\begin{proof}
It readily follows from \eqref{r12} using an induction
  argument that $h\in
buc^{2+\alpha}(\Gamma_0)$ with $h=\sum_{n\in\Z} \hat h_n e^{in\theta}$
belongs to the kernel of $DH(0)$ if and only if $\hat h_n=0$ for
$n\in\Z\setminus\{\pm 1\}$. Hence the kernel of $DH(0)$ is spanned by
$\{ e^{-i\theta}, e^{i\theta}\}$. Next,
$$
DH(0)= A+\mu B,
$$
with
$-A\in\mathcal{H}\big(buc^{2+\alpha}(\Gamma_0),buc^{1+\alpha}(\Gamma_0)\big)$
(see Theorem~\ref{T1} with $\mu=0$ therein) given by 
\begin{equation}\label{r12a}
\begin{split}
A h& :=-\frac{12\mathcal{V}a}{\pi R_0^4}\, \hat h_0 -
\frac{4\mathcal{V}a}{\pi R_0^4}\sum_{n\in\Z^*}\big(\vert n\vert
-1\big)\,\hat h_n \, e^{in\theta}\end{split} 
\end{equation}
and $B\in\mathcal{L}\big(buc^{2+\alpha}(\Gamma_0),buc^{1+\alpha}(\Gamma_0)\big)$ given by
\begin{equation}\label{r12b}
\begin{split}
Bh:= - \frac{aR_0}{8}\sum_{n\in\Z^*}\left[\big(\vert n\vert +n-4\big)
  \, \hat h_{n-1} + \big(\vert n\vert -n-4\big) \, \hat
  h_{n+1}\right]\,   e^{in\theta}. 
\end{split}
\end{equation}
Set $\omega:=\frac{\mathcal{V}a}{\pi R_0^4}$. Since $\sigma(A)=
-4\omega\mathbb{N}$, \cite[I.Corollary 1.4.3]{LQPP} implies that there
are $\mu_0>0$ and $\vartheta\in(0,\pi/2)$ such that
$\sigma(A+\mu B)\subset \omega+\Sigma_\vartheta$ for $\mu\in [0,\mu_0)$, 
where $\Sigma_\vartheta:=[\arg \left(z\right) \in (\pi-\vartheta,\pi+\vartheta)]$.
Since zero is the only eigenvalue of $A$ in $\Sigma_\vartheta \cap
[\mathrm{Re}\, z> -\omega]$ and  since $A+\mu B\rightarrow A$ in the
generalized sense of \cite[IV.Theorem 2.24]{Kato} as $\mu\rightarrow
0$, it follows from \cite[IV.Theorem 3.16]{Kato} that zero is the only
eigenvalue of  $DH(0)= A+\mu B$ in $\Sigma_\vartheta \cap
[\mathrm{Re}\, z> -\omega]$ if $\mu\ge 0$ is sufficiently small. This
proves the assertion. 
\end{proof}

\subsection{Stability Analysis}

To analyze now the stability of the zero solution to \eqref{rhoevd} we
need to split off the zero eigenvalue of the linearization
$DH(0)$. For this it is useful to use a slightly different description
of the curves $\Gamma_\rho$ as provided by the next lemma. 

\begin{lem}\label{L5}
For each $\rho\in \mathcal{O}$ there is a unique $(z,\bar \rho)\in
\R^2\times \big(\mathrm{ker}(DH(0))\big)^\perp$ such that
$\Gamma_\rho=z+\Gamma_{\bar \rho}$. Moreover, $H(\rho)= H(\bar
\rho)$. 
\end{lem}

\begin{proof}
Since the kernel of $DH(0)$ is spanned by $\{ e^{-i\theta},
e^{i\theta}\}$ {according to} Lemma~\ref{L51}, the existence of a unique $(z,\bar
\rho)\in \R^2\times \left(\mathrm{ker}(DH(0))\right)^\perp$ with
$\Gamma_\rho=z+\Gamma_{\bar \rho}$ is shown in \cite[Lemma
5.2]{PG1}. That $H(\rho)=H(\bar \rho)$ follows as in \cite[Lemma
5.1]{PG1} from the translation invariance of the problem. Indeed, if
$\Gamma=\partial\Omega$ and $u_\Gamma$ solves 
\begin{alignat*}{2}
  -\Delta u&=\mu x^1+\lambda\quad && \text{in }\Omega,\\
  u&=0&&\text{on }\Gamma,\\
  \int _{\Omega}u\, dx&=\mathcal{V},&&
 \end{alignat*}
then $u_{\Gamma_{\bar\rho}}=u_{\Gamma_{\rho}}(\cdot +z)$ on
$\Omega_{\bar\rho}=-z+\Omega_\rho$ and
$\nu_{\Gamma_{\bar\rho}}=\nu_{\Gamma_{\rho}} (\cdot +z)$ on
$\Gamma_{\bar\rho}=-z+\Gamma_\rho$. The definition of $H$ implies the
assertion. 
\end{proof}

We next derive the evolution in the new coordinates $(z,\bar
\rho)=(z(\rho),\bar \rho(\rho))$. Let $\nu_0=(\cos,\sin)$ and
$\tau_0=(-\sin,\cos)$ denote the unit normal and unit
tangent vector to $\Gamma_0$, respectively,
and let
$$ 
p=(R_0+\rho(t,\theta))\nu_0(\theta)= z+(R_0+\bar\rho(t,\phi))\nu_0(\phi)
$$
be an arbitrary point on $\Gamma_{\rho(t,\cdot)}=z+
\Gamma_{\bar\rho(t,\cdot)}$. We often suppress the time variable $t$
in the following. Differentiating with respect to $\phi$ implies that
\begin{equation}\label{107}
 \nu_{\Gamma_{\rho}}(p)=\frac{-\partial_\phi\bar\rho(\phi)\tau_0(\phi)
 +(R_0+\bar\rho(\phi))\nu_0(\phi)}{\big((\partial_\phi\bar
 \rho(\phi))^2+(R_0+\bar\rho(\phi))^2\big)^{1/2}}.
\end{equation}
Now, since $V_{\Gamma_\rho}(p)=\big(\dot
z+\dot{\bar\rho}(\phi)\nu_0(\phi)\big)\cdot \nu_\rho(p)$ with {dots}
indicating time derivatives, it follows from the definition of $H$ and
$H(\bar\rho)=H(\rho)$  that
$$
H(\bar\rho)(p)=\frac{1}{\nu_0(\phi)\cdot\nu_\rho(p)} \big(\dot
z+\dot{\bar\rho}(\phi)\nu_0(\phi)\big)\cdot \nu_\rho(p).
$$
Therefore, from \eqref{107} we derive that
\begin{equation}\label{110}
\begin{split}
H(\bar\rho)=\left(\frac{\partial_\phi
    \bar\rho(\phi)}{R_0+\bar\rho(\phi)} \sin\phi  +\cos\phi\right)\dot
z_1 +\left(-\frac{\partial_\phi \bar\rho(\phi)}{R_0+\bar\rho(\phi)}
  \cos\phi +\sin\phi\right) \dot z_2+  \dot{\bar\rho}.
\end{split}
\end{equation}
Let $\Pi^{(\pm 1)}\in\mathcal{L}(buc^{1+\alpha}(\Gamma_0))$ denote the
projection onto the subspace spanned by $e^{\pm i\theta}$, that is, 
$$
\Pi^{(\pm 1)} f :=\left(\frac{1}{2\pi}\int_0^{2\pi} f(\vartheta)e^{\mp
    i\vartheta}\, d\vartheta\right) e^{\pm i\theta}, 
$$
and let $\Pi^\perp\in\mathcal{L}(buc^{1+\alpha}(\Gamma_0))$  denote
the projection onto $\big(\mathrm{ker}(DH(0))\big)^\perp$. We then
apply these projections to \eqref{110} to derive the evolution for
$(z,\bar\rho)$. For a more compact notation we introduce 
$$
M(\bar\rho):=\left[\begin{matrix} \frac{1}{2} & \frac{1}{2i}\\
    \frac{1}{2} & -\frac{1}{2i}\end{matrix} \right] + 
\left[\begin{matrix} \frac{1}{4\pi i}\int_0^{2\pi} \frac{\partial_\phi
      \bar\rho(\phi)}{R_0+\bar\rho(\phi)}\left(1-e^{-2\phi i}\right)
    d\phi  & -\frac{1}{4\pi }\int_0^{2\pi} \frac{\partial_\phi
      \bar\rho(\phi)}{R_0+\bar\rho(\phi)}\left(e^{-2\phi i}+1\right)
    d\phi\\
\frac{1}{4\pi i}\int_0^{2\pi} \frac{\partial_\phi
  \bar\rho(\phi)}{R_0+\bar\rho(\phi)}\left(e^{2\phi i}-1\right) d\phi
& -\frac{1}{4\pi}\int_0^{2\pi} \frac{\partial_\phi
  \bar\rho(\phi)}{R_0+\bar\rho(\phi)}\left(e^{2\phi i}+1\right)
d\phi \end{matrix} \right] 
$$
and observe that 
$$
M(\bar\rho)= \left[\begin{matrix} \frac{1}{2} & \frac{1}{2i}\\
    \frac{1}{2} & -\frac{1}{2i}\end{matrix} \right] +
O\bigl(\|\bar\rho\|_{buc^{2+\alpha}(\Gamma_0)}\bigr)
$$
is invertible if $\bar\rho\in\mathcal{O}$ is small. Further set
$$
\Pi (H(\bar\rho)):=\left(\begin{matrix} \Pi^{(+1)} (H(\bar\rho)) \\
    \Pi^{(-1)} (H(\bar\rho))\end{matrix}\right) 
$$
and
$$
f(\bar\rho):=\Pi^\perp\left( \frac{\partial_\phi
    \bar\rho}{R_0+\bar\rho} \tau_0 \cdot M(\bar\rho)^{-1} \Pi
  (H(\bar\rho))\right).
$$
Then we obtain from \eqref{110}: 

\begin{prop}
The evolution for $\rho$ governed by \eqref{rhoevd} is equivalent to the system 
\begin{align}
\dot z&= M(\bar\rho)^{-1} \Pi (H(\bar\rho)),\label{z}\\
\dot{\bar\rho}&= \Pi^\perp H(\bar\rho) +f(\bar\rho),\label{barrho}
\end{align}
for $(z,\bar\rho)$.
\end{prop}
Next, we introduce
$$
buc_\perp^{k+\alpha}(\Gamma_0):=\Pi^\perp buc^{k+\alpha}(\Gamma_0),\quad k=1,2.
$$
Then $\mathcal{O}_\perp:=\mathcal{O}\cap buc_\perp^{2+\alpha}(\Gamma_0)$ is an open zero-neighborhood in $buc_\perp^{2+\alpha}(\Gamma_0)$ and \eqref{H} entails that
\begin{equation}\label{114}
f\in C^2\big(\mathcal{O}_\perp,buc_\perp^{1+\alpha}(\Gamma_0)\big)\quad  \text{with}\quad f(0)=0,\ D f(0)=0.
\end{equation}
Moreover, for the linearization of $\Pi^\perp H\in C^2\big(\mathcal{O}_\perp, buc_\perp^{1+\alpha}(\Gamma_0)\big)$ at zero we have:

\begin{prop}\label{P44}
Let $\mu\ge 0$ be sufficiently small. Then
$$
-D\big(\Pi^\perp H\big)(0)\in\mathcal{H}\big(buc_\perp^{2+\alpha}(\Gamma_0),buc_\perp^{1+\alpha}(\Gamma_0)\big),
$$ 
and there is $\omega_0>0$ such that
$$
\sigma\big(D\big(\Pi^\perp H\big)(0)\big)\subset [\mathrm{Re} z\le -2\omega_0].
$$
\end{prop}

\begin{proof} Observe that, owing to \eqref{r12a}, \eqref{r12b},
$$
D\big(\Pi^\perp H\big)(0)= \Pi^\perp DH(0) =A_\perp +\mu B_\perp,
$$
where
\begin{equation*}
\begin{split}
A_\perp h& :=-\frac{12\mathcal{V}a}{\pi R_0^4}\, \hat h_0 - \frac{4\mathcal{V}a}{\pi R_0^4}\sum_{\substack{ n\in\Z^*\\ n\not=\pm 1}}\big(\vert n\vert -1\big)\,\hat h_n \, e^{in\theta}
\end{split}
\end{equation*}
and 
\begin{equation*}
\begin{split}
B_\perp h:= - \frac{aR_0}{8}\sum_{\substack{ n\in\Z^*\\ n\not=\pm 1}}\left[\big(\vert n\vert +n-4\big) \, \hat h_{n-1} + \big(\vert n\vert -n-4\big) \, \hat h_{n+1}\right]\,   e^{in\theta}
\end{split}
\end{equation*}
for $h\in buc_\perp^{2+\alpha}(\Gamma_0)$ with
$h(\theta)=\sum\limits_{n\in\Z,\, n\not=\pm 1} \hat h_n e^{in\theta}.$
Since $-A\in\mathcal{H}\big(buc^{2+\alpha}(\Gamma_0),buc^{1+\alpha}(\Gamma_0)\big)$ by Theorem~\ref{T1} (with $\mu=0$), one readily deduces from the obvious fact
$$
A_\perp=A\big\vert_{buc_\perp^{2+\alpha}(\Gamma_0)}
$$
that
$-A_\perp\in\mathcal{H}\big(buc_\perp^{2+\alpha}(\Gamma_0),buc_\perp^{1+\alpha}(\Gamma_0)\big)$. Observing then that
$$
\sigma(A_\perp)=\sigma(A)\setminus\{0\} =-4\omega\mathbb{N}\setminus\{0\},\quad \omega:=\frac{\mathcal{V}a}{\pi R_0^4},
$$
the assertion follows from the perturbation result \cite[I.Proposition 1.4.2]{LQPP}.
\end{proof}

Due to Proposition~\ref{P44} and \eqref{114} we are in a position to
apply the principle of linearized stability from \cite[Theorem
9.1.2]{Lunardi} to the equation \eqref{barrho} for $\bar\rho$. Thus,
there are $r>0$ and $M>0$ such that for any initial value
$\bar\rho_0\in buc_\perp^{2+\alpha}(\Gamma_0)$ with
$\|\bar\rho_0\|_{buc_\perp^{2+\alpha}}<r$, the solution $\bar\rho$ to
\eqref{barrho} with $\bar\rho(0)=\bar\rho_0$ exists globally in time
and 
\begin{equation}\label{z1}
\|\bar\rho(t)\|_{buc_\perp^{2+\alpha}}+\|\dot{\bar\rho}(t)\|_{buc_\perp^{1+\alpha}}
\le Me^{-\omega_0 t},\quad t\ge 0. 
\end{equation}
Plugging this into \eqref{z} and observing that the right-hand side of
\eqref{z} is of order
$O\bigl(\|\bar\rho\|_{buc_\perp^{2+\alpha}}\bigr)$ owing
to \eqref{H}, it readily follows that there is $z_\infty\in\R^2$ such that 
\begin{equation}\label{z2}
\vert z(t)-z_\infty\vert\le c e^{-\omega_0 t},\quad t\ge 0.
\end{equation}
Therefore, using again the original coordinates instead of
$(z,\bar\rho)$ and noticing that
$\|\bar\rho(\rho)\|_{buc_\perp^{2+\alpha}}$ is small if and only if
$\|\rho\|_{buc^{2+\alpha}}$ is,  we arrive at:

\begin{thm}\label{T4}
Assume that \eqref{affine} holds. Then, for small
incline $\mu>0$, the stationary solution $(u_0,\Omega_*)$ to
\eqref{d1c}-\eqref{d5c} from \eqref{ui} is asymptotically
exponentially stable. More precisely, if $\mu>0$ is small and given an
initial geometry
$$\Gamma_{\rho_0}=\{(R_0+\rho_0(\theta))e^{i\theta}\} $$
with $\|\rho_0\|_{buc^{2+\alpha}(\Gamma_0)}$ sufficiently small, there exists
$$
\rho\in C^1\big(\R^+,buc^{1+\alpha}(\Gamma_0)\big)\cap
C\big(\R^+,buc^{2+\alpha}(\Gamma_0)\big) 
$$
and $u_\rho$ with
$$
u_{\rho(t)}\in
buc^{1+\alpha}\big(\Omega_{\rho(t)})\big),\quad \partial\Omega_{\rho(t)}=\Gamma_{\rho(t)}
$$
for $t\ge 0$
such that $(\rho,u_\rho)$ satisfies \eqref{rhoevd}. Moreover,
$$
\Gamma_{\rho(t)}=z(t)+\Gamma_{\bar\rho(t)},\quad t\ge 0,
$$
 with
$$
\bar\rho\in C^1\big(\R^+,buc^{1+\alpha}(\Gamma_0)\big)\cap
C\big(\R^+,buc^{2+\alpha}(\Gamma_0)\big) ,\qquad   z\in
C^1\big(\R^+,\R^2\big) 
$$
satisfying \eqref{z1} and \eqref{z2}.
\end{thm}

The above theorem shows that,  for small inclines and
when starting out with a droplet geometry sufficiently close to the
disk of radius $R_0$, the droplet asymptotically becomes circular of
radius $R_0$ sliding down the plane with constant speed $v_0$.


\end{document}